\documentclass[a4paper,11pt]{amsart}
\usepackage{amsmath,amssymb,amsfonts,mathrsfs}
\newtheorem{mydef}{Definition}[section]
\newtheorem{thm}{Theorem}[section]
\newtheorem{rmk}{Remark}[section]

\newtheorem{lemma}{Lemma}[section]
\newtheorem{propo}{Proposition}[section]

\numberwithin{equation}{section}
\usepackage{amssymb}
\usepackage{amsmath}
\usepackage{rotating}
\usepackage{chngpage}
\usepackage{booktabs}
\usepackage{mathtools}
\usepackage{float}
\usepackage{breqn}

\def\N{{\rm I\kern-0.16em N}}
\def\R{{\rm I\kern-0.16em R}}
\def\E{{\rm I\kern-0.16em E}}
\def\P{{\rm I\kern-0.16em P}}
\def\F{{\rm I\kern-0.16em F}}
\def\B{{\rm I\kern-0.16em B}}
\def\C{{\rm I\kern-0.46em C}}
\def\G{{\rm I\kern-0.50em G}}
\newcommand{\ud}{\mathrm{d}}
\numberwithin{equation}{section}
\font\eka=cmex10
\usepackage{ae}
\def\ind{\mathrel{\hbox{\rlap{%
\hbox to 7.5pt{\hrulefill}}\raise6.6pt\hbox{\eka\char'167}}}}
\begin{document}
\title[Parameter estimation]
{Drift parameter estimation for \\fractional Ornstein-Uhlenbeck process \\of the Second Kind}
\author[Azmoodeh and Morlanes]{Ehsan Azmoodeh \and Jos\'e Igor Morlanes}
\address{Jos\'e Igor Morlanes \\ Statistiska institutionen, Stockholms Universiet\\
P.O. Box SE-10691 STOCKHOLM.} 
\email{jose.morlanes@stat.su.se}

\address{Ehsan Azmoodeh\\ Facult\'e des Sciences, de la Technologie et de la Communication,  Universit\'e du Luxembourg\\
P.O. Box L-1359 LUXEMBOURG.}
\email{ehsan.azmoodeh@uni.lu}
\begin{abstract}
 Fractional Ornstein-Uhlenbeck process of the second kind $(\text{fOU}_{2})$ is solution of the Langevin equation $\mathrm{d}X_t = -\theta X_t\,\mathrm{d}t+\mathrm{d}Y_t^{(1)}, \ \theta >0$ with driving noise 
$ Y_t^{(1)} := \int^t_0 e^{-s} \,\mathrm{d}B_{a_s}; \ a_t= H e^{\frac{t}{H}}$ where $B$ is a fractional Brownian motion with Hurst parameter $H \in (0,1)$. In this article, in the case $H>\frac{1}{2}$, we prove that the least squares estimator 
$\widehat{\theta}_T$ introduced in [\cite{h-n}, Statist. Probab. Lett. 80, no. 11-12, 1030-1038], provides a consistent estimator. Moreover, using central limit theorem for multiple Wiener integrals, we prove asymptotic normality 
of the estimator valid for the whole range $H \in(\frac{1}{2},1)$.

\medskip


\smallskip

\end{abstract}

\maketitle

\section{Introduction}
Assume $B=\{B_t\}_{ t\geq 0}$ is a fractional Brownian motion with Hurst parameter $H \in (0,1)$, i.e. a continuous, centered Gaussian process with covariance function
\begin{equation*}
 R_{H}(t,s) = \frac{1}{2}\{ t^{2H}+s^{2H} - |t-s|^{2H}\}.
\end{equation*}
Consider Gaussian process $ Y_t^{(1)} := \int^t_0 e^{-s} \,\mathrm{d}B_{a_s}$ with $a_t= H e^{\frac{t}{H}}$. The fractional Ornstein-Uhlenbeck process of the second kind $X$ with initial value $X_0$ is the solution of the Langevin equation 
\begin{equation}\label{LfOU2}
 \mathrm{d}X_t = -\theta X_t\,\mathrm{d}t+\mathrm{d}Y_t^{(1)},\quad   X_0= X_0, \quad \theta >0.
\end{equation}
The terminology \textit{``of the second kind''} is taken from Kaarakka and Salminen \cite{k-s}. The motivation behind the process $X$ is that it is related to Lamperti transformation of fractional Brownian motion. For more detailed information 
on fractional Ornstein-Uhlenbeck processes, see Subsection \ref{fOU}.\\ 
 
An interesting problem in mathematical statistics is to estimate the unknown parameter $\theta$ based on continuous observation of the sample paths of the process $X_t, \ t \in [0,T]$. In the case of fractional Ornstein-Uhlenbeck process 
of the first kind $(\text{fOU}_{1})$, that is replace driving noise $Y^{(1)}$ with fractional Brownian motion $B$ in the Langevin equation $(\ref{LfOU2})$, when $H > \frac{1}{2}$, Hu and Nualart \cite{h-n} suggested 
an estimator so called \textit{the least squares estimator} defined as
\begin{equation}\label{leastsquare}
\widehat \theta _T =-\frac{\int^T_0 X_t\,\delta X_t}{\int^T_0 X^2_t\,\ud t}
\end{equation}
where the stochastic integral is interpreted as Skorokhod integral $($see Subsection \ref{Malliavin} for definitions and notations$)$. The motivation comes from the 
following heuristic argument. The least squares estimator is obtain by minimizing the function
\begin{equation*}
 \theta \mapsto \int^T_0 \vert \dot{X_t} + \theta X_t \vert ^2 \ud t
\end{equation*}
which leads to the solution $(\ref{leastsquare})$. In the same setup as fractional Ornstein-Uhlenbeck process of the first kind, the maximum likelihood estimator of drift parameter $\theta$ is found by Kleptsyna and Le Breton \cite{k-l} and its 
strong consistency is proved. Estimation of drift parameter $\theta$ for fractional Ornstein-Uhlenbeck process of the first kind, in the non-ergodic case, i.e. $\theta<0$, is studied by Belfadli, Es-Sebaiy and Ouknine in \cite{b-e-o}.
In this paper, we consider the least squares estimator $\widehat{\theta}_T$ in the setup of the fractional Ornstein-Uhlenbeck process of the second kind. The key point is the Lemma \ref{alternative_X_t}. Using this lemma, we 
replace the Gaussian noise $Y^{(1)}$ with an equivalent $($in distribution$)$ noise and do computations in an equivalent model.\\

The paper is organized as follows. In the section \ref{prel}, we give a detailed information on fractional Ornstein-Uhlenbeck processes and Malliavin calculus for factional Brownian motion. Section \ref{main} is devoted to our main results. In the 
appendix section, we provide some auxiliary computation which is used in the proof of main results.

\section{Preliminaries}\label{prel}
It is well known that the classical Ornstein-Uhlenbeck process $U^{(\frac{1}{2},\xi_{0})}=\{ U_{t}^{(\frac{1}{2},\xi_{0})}\}_{t\geq 0}$ with initial value $\xi_0$ 
can be constructed as the unique solution of the Langevin SDE
\begin{equation}\label{Langevin_B}
 dU_{t}^{(\frac{1}{2},\xi_{0})} = -\theta U_{t}^{(\frac{1}{2},\xi_{0})} dt + dW_t, \quad t \ge 0,
\end{equation}
where $\theta > 0 $ and $W=\{W_t\}_{ t \ge 0}$ is standard Brownian motion. The solution of SDE (\ref{Langevin_B}) can be expressed as 
\begin{equation*}
U_{t}^{(\frac{1}{2},\xi_{0})}:= e^{- \theta t} \left( \xi_{0} + \int_{0}^{t} e^{\theta s} \ud W_s\right), \quad t\ge 0.
\end{equation*}

Let $\hat{W}$ denote a two sided Brownian motion defined as
\begin{equation}\label{two-sided}
 \hat W_t := \begin{cases}
  W_t & \text{for $t\geq 0$} \\
 W^{(-)}_{-t}  & \text{for $t\le 0$},
 \end{cases}
\end{equation}
where $W^{(-)}_t=\{W^{(-)}_t; t\geq 0\}$ is another Brownian motion initiated at 0 and independent of $W$. The selection 
$\xi_{0}= \int_{-\infty}^{0} e^{\theta s} \ud \hat{W}_s$ leads to the unique solution $U^{(\frac{1}{2})}$ given by 
\begin{equation*}
U_{t}^{(\frac{1}{2})}:= \int_{-\infty}^{t} e^{- \theta (t-s)} \ud \hat{W}_s, \quad t\ge 0.
\end{equation*}
It is a stationary, continuous Gaussian process with covariance function 
\begin{equation*}
\text{Cov} ( U_{s}^{(\frac{1}{2})}, U_{t}^{(\frac{1}{2})} ) = \frac{1}{2 \theta} e^{- \theta |t-s|}, \quad \forall s,t \ge 0.
\end{equation*}
On the other hand, it is known that the classical Ornstein-Uhlenbeck process $ U^{(\frac{1}{2})}$ can be reconstructed from Brownian motion $W$ by Lamperti transformation. For $\alpha > 0$, 
define the process
\begin{equation*}
Z_{t}^{(\frac{1}{2})} := e^{- \theta t} W_{\alpha e^{2 \theta t}}, \quad t \in \R.
\end{equation*}
Then $Z^{(\frac{1}{2})}$ is a stationary, Gaussian process with covariance function
\begin{equation*}
\text{Cov}(Z_{t}^{(\frac{1}{2})},Z_{s}^{(\frac{1}{2})}) = \alpha e^{- \theta |t-s|}, \quad s,t \in \R.
\end{equation*}
Since, finite dimensional distributions of a Gaussian process are completely characterized by its mean and covariance functions, therefore with
$\alpha= \frac{1}{2 \theta}$, we have $U^{(\frac{1}{2})} \stackrel{\text{law}}{=} Z^{(\frac{1}{2})} $ where $\stackrel{\text{law}}{=}$ stands for equality in law. For more information 
on Lamperti transformation and related topics we refer to the book \cite{e-m}.

\subsection{Fractional Ornstein-Uhlenbeck processes}\label{fOU}

In this subsection, we briefly introduce fractional Ornstein-Uhlenbeck processes. The main references are \cite{C-K-M} and \cite{k-s}. We mainly focus on fractional Ornstein-Uhlenbeck process of the second kind that is the core
stochastic process of the article. To obtain fractional 
Ornstein-Uhlenbeck processes, replace the Brownian motion $W$ with the fractional Brownian motion $B$ in Langevin equation (\ref{Langevin_B}). Consider the following 
stochastic differential equation

\begin{equation}\label{Langevin-first}
 dU^{(H,\xi_0)}_t =-\theta U^{(H,\xi_0)}_t dt + dB_t
\end{equation}

with initial value $\xi_0$, The solution of the following SDE can be expressed as

\begin{equation}\label{generalfOU1}
U^{(H,\xi_0)}_t = e^{-\theta t} \left( \xi_{0} + \int^t_0 e^{\theta s}\,\ud B_s\right).
\end{equation}

Note that the stochastic integral is understood as a path-wise Riemann-Stieltjes integral. Let $\hat{B}$ denote a two sided fractional Brownian motion $($see \ref{two-sided}$)$. The selection
\begin{equation*}
\xi_{0} := \int^0_{-\infty} e^{\theta s}\,\ud \hat B_s
\end{equation*}
for the initial value $\xi_0$ leads to an unique stationary, Gaussian process $U^{(H)}$ of the form 

\begin{equation}\label{U^{(H)}}
U^{(H)}_t = e^{-\theta t} \int^t_{-\infty} e^{\theta s}\,\ud \hat B_s.
\end{equation}

\begin{mydef}
We call the process $U^{(H,\xi_0)}$ given by $(\ref{generalfOU1})$ a fractional Ornstein-Uhlenbeck process of the first kind with initial value $\xi_0$. The process $U^{(H)}$ defined in $(\ref{U^{(H)}})$ is called 
stationary fractional Ornstein-Uhlenbeck process of the first kind.
\end{mydef}

\begin{rmk}
 It is shown in \cite{C-K-M} that the covariance function of the stationary process $U^{(H)}$ decays like a power function, so it is ergodic and for $H \in (\frac{1}{2},1)$, it 
exhibits long range dependence.
\end{rmk}

Now, we define a new stationary, Gaussian process $X^{(\alpha)}$ by means of Lamperti transformation of fractional Brownian motion $B$:
\begin{equation*}
 X_t^{(\alpha)} := e^{-\alpha t}B_{a_t},\quad t\in \R,
\end{equation*}
where $\alpha>0$ and $a_t= \frac{H}{\alpha}e^{\frac{\alpha t}{H}}$. We aim to represent the process $X^{(\alpha)}$ as solution of a Langevin type SDE. For this reason, consider 
the process $Y^{\alpha}$ defined via
\begin{equation*}
 Y_t^{(\alpha)} := \int^t_0 e^{-\alpha s} \,\mathrm{d}B_{a_s}, \quad t \ge 0.
\end{equation*} 
The stochastic integral is understood in as path-wise Riemann-Stieltjes integral. Using the self-similarity property of fractional Brownian motion one can see that the process $Y^{(\alpha)}$ satisfies in the following scaling property

\begin{equation}\label{scaling}
 \{ Y^{(\alpha)}_{t / \alpha} \}_{t \ge 0} \stackrel{\text{law}}{=} \{ \alpha^{-H} Y^{(1)}_{t}\}_{t \ge 0}.
\end{equation}
Using $Y^{(\alpha)}$ the process $ X^{(\alpha)}$ can be viewed as the solution of the following Langevin type SDE
\begin{equation*}
 \mathrm{d}X_t^{(\alpha)} = -\alpha X_t^{(\alpha)}\,\mathrm{d}t+\mathrm{d}Y_t^{(\alpha)},
\end{equation*}
with random initial value $X_0^{(\alpha)}=B_{a_0}\stackrel{d}{=}B_{H/\alpha}\sim N(0, (\frac{H}{\alpha})^{2H})$.\\

Inspired by the scaling property $($\ref{scaling}$)$, we consider the following Langevin equation with $Y^{(1)}$ as the driving noise:
\begin{equation}
\label{Langevin}
 \mathrm{d}X_t = -\theta X_t\,\mathrm{d}t+\mathrm{d}Y_t^{(1)},\qquad   \theta > 0.
\end{equation}



The solution of the SDE $(\ref{Langevin})$ is given by 
\begin{equation}\label{generalfOU2}
X_t = e^{- \theta t} \left( X_0 + \int_{0}^{t} e^{\theta s} \,\mathrm{d} Y^{(1)}_s \right) = e^{- \theta t} \left( X_0 + \int_{0}^{t} e^{(\theta -1)s} \,\mathrm{d} B_{a_s} \right)
\end{equation}
with $\alpha=1$ in $a_t$. Notice that the stochastic integral is understood as path-wise Riemann-Stieltjes integral. The selection $X_0 = \int^0_{-\infty} e^{(\theta-1) s} \,\mathrm{d}B_{a_s}$ leads to an unique stationary, Gaussian process 
\begin{equation}\label{U}
 U_t= e^{-\theta t}\int^t_{-\infty} e^{(\theta-1) s} \,\mathrm{d}B_{a_s}.
\end{equation}

\begin{mydef}
We call the process $X$ given by $(\ref{generalfOU2})$ a fractional Ornstein-Uhlenbeck process of the second kind with initial value $X_0$. The process $U$ defined in $(\ref{U})$ is called the stationary fractional Ornstein-Uhlenbeck process of the second kind.
\end{mydef}

\begin{propo}\cite{k-s}
The covariance function of the stationary process $U$ decays exponentially and has short range dependence.
\end{propo}


\subsection{Malliavin calculus with respect to fractional Brownian motion}\label{Malliavin}
In this subsection, we briefly introduce some basic facts on Malliavin calculus with respect to fractional Brownian motion. We also recall some required results 
related to fractional Brownian motion which we need for the proof of our main theorem. The main references are \cite{a-m-n}, \cite{n-p} and \cite{nu1}.\\

Assume $B=\{ B_t \}_{t\in [0,T]}$ is a fractional Brownian motion with Hurst parameter $H \in (0,1)$ and covariance function $R_{H}(t,s)$. It is well-known that the covariance function $R_H$ can be represented as
\begin{equation*}
 R_{H}(t,s) = \int_{0}^{t \wedge s} K_{H}(t,u) \ K_{H}(s,u) \ud u
\end{equation*}
for a Volterra-type square integrable kernel $K_H$. In the case $H>\frac{1}{2}$, the kernel $K_H$ has a simple expression given by
\begin{equation*}
K_{H}(t,s)= c_H s^{\frac{1}{2} - H} \int_{s}^{t} (u - s)^{H - \frac{3}{2}} u^{H - \frac{1}{2}} \ud u
\end{equation*}
where $c_H = (H - \frac{1}{2}) \Big( \frac{2H \Gamma(\frac{3}{2} - H)}{ \Gamma(H + \frac{1}{2}) \Gamma(2- 2H)} \Big)^{\frac{1}{2}}$.
Consider the set $\mathcal{E}$ of all step functions on $[0,T]$. We define the Hilbert space $\mathcal{H}$ as the closure of $\mathcal{E}$ with respect
to inner product
\begin{equation*}
 \langle \mathbf{1}_{[0,t]},\mathbf{1}_{[0,s]} \rangle_{\mathcal{H}} = R_{H}(t,s).
\end{equation*}
Then, the mapping $\mathbf{1}_{[0,t]} \mapsto B_t$ can be extended to an isometry between Hilbert space $\mathcal{H}$ and Gaussian space $\mathcal{H}_1$ associated with 
fractional Brownian motion $B$. We denote this isometry by $\varphi \mapsto B(\varphi)$. It is known that when $H= \frac{1}{2}$, we have $\mathcal{H}=L^{2}[0,T]$, 
whereas for $H>\frac{1}{2}$, the elements of $\mathcal{H}$ may be not functions but distributions of negative order $($see \cite{p-t}$)$. An interesting subspace
$\mathcal{|H|}$ containing only functions is the set of all measurable functions $\varphi$ on $[0,T]$ such that
\begin{equation*}
 \Vert \varphi \Vert_{|\mathcal{H}|} ^{2} := \alpha_H \int_{0}^{T} \int_{0}^{T} \vert \varphi (s)\vert  \vert \varphi(t)\vert  |t-s|^{2H-2} \ud s \ud t < \infty,
\end{equation*}
where $\alpha_H = H(2H-1)$. Notice that for any two measurable functions $\varphi, \psi \in |\mathcal{H}|$, we have
\begin{equation}\label{isometryfbm}
 \E \big( B(\varphi) B(\psi) \big)= \langle \varphi, \psi \rangle _{\mathcal{H}} = \alpha_H \int_{0}^{T} \int_{0}^{T} \varphi (s) \psi(t) |t-s|^{2H-2} \ud s \ud t.
\end{equation}
 Notice that when $H>\frac{1}{2}$, we have the inclusions  $L^{2}[0,T] \subset L^{\frac{1}{H}}[0,T] \subset |\mathcal{H}| \subset \mathcal{H}$. Consider the linear 
operator $K^{*}_{H}$ between $\mathcal{E}$ and $L^{2}[0,T]$ defined by
\begin{equation*}
 (K^{*}_{H}\varphi)(s)= \int_{s}^{T} \varphi(t) \frac{\partial K_H}{\partial t} (t,s) \ud t.
\end{equation*}
Notice that the operator $K^{*}_{H}$ is an isometry between $\mathcal{E}$ and $L^{2}[0,T]$ that can be extended to the Hilbert space $\mathcal{H}$. Moreover, the processes
$W= \{ W_t \}_{t \in [0,T]}$ given by
\begin{equation}\label{BW}
 W_t := B \big( (K^{*}_{H})^{-1} (\textbf{1}_{[0,t]})\big)
\end{equation}
defines a Brownian motion. The fractional Brownian motion $B$ and Brownian motion $W$ relate through the integral representation
\begin{equation}\label{representation FBW}
B_t = \int_{0}^{t} K_{H}(t,s) \ud W_s.
\end{equation}


Consider the space $\mathcal{S}$ of all smooth random variables of the form 
\begin{equation}\label{eq:smooth}
F= f(B(\varphi_1), \cdots, B(\varphi_n)), \qquad \varphi_1, \cdots, \varphi_n \in \mathcal{H},
\end{equation}
where $f \in C_{b}^{\infty}(\R^n)$. For any smooth random variable $F$ of the form \ref{eq:smooth}, we define its Malliavin derivative $D^{(B)}= D $ as an element of 
$L^{2}(\Omega;\mathcal{H})$ by
\begin{equation*}
D F= \sum_{i=1}^{n} \partial_{i} f (B(\varphi_1), \cdots, B(\varphi_n)) \varphi_i.
\end{equation*}
In particular, $D B_t = \mathbf{1}_{[0,t]}$. We denote by $\mathbb{D}^{1,2}_{B}= \mathbb{D}^{1,2}$ the Hilbert space of all square integrable Malliavin derivative random variables as the 
closure of the set $\mathcal{S}$ of smooth random variables with respect to norm 
\begin{equation*}
\Vert F \Vert_{1,2}^{2} = \E |F|^{2} + \E ( \Vert D F \Vert_{\mathcal{H}} ^{2}).
\end{equation*}
The \textit{transfer principle} connects the Malliavin operators of the both processes $B$ and $W$. This is the message of the next proposition.
\begin{propo}\label{transfer}\cite{nu1}
 For any $F \in \mathbb{D}_{W}^{1,2}= \mathbb{D}^{1,2}$,
\begin{equation*}
 K_{H}^{*} D F = D^{(W)} F,
\end{equation*}
where $D^{(W)}$ denotes the Malliavin derivative operator with respect to Brownian motion $W$, and $\mathbb{D}_{W}^{1,2}$ the corresponding Hilbert space.
\end{propo}

The \textit{divergence operator} denoted by $\delta_{B} = \delta$ is the adjoint operator of Malliavin derivative operator $D$. For any $\mathcal{H}$-valued random variable 
$u \in L^{2}(\Omega;\mathcal{H})$ belongs to the domain $\text{Dom} \delta$ of the divergence operator $\delta$, then the random variable $\delta(u)$ is defined by
duality relationship
\begin{equation*}
 \E (F \delta(u)) = \E \langle D F, u \rangle_{\mathcal{H}}, \qquad \forall F \in \mathbb{D}^{1,2}.
\end{equation*}
 Also, an element $u \in L^{2}(\Omega;\mathcal{H})$ belongs to the domain $\text{Dom} \delta$ if and only if
\begin{equation*}
\big| \E \langle D F, u \rangle_{\mathcal{H}} \big| \le c_{u} \Vert F \Vert_{L^{2}}
\end{equation*}
for any $F \in \mathbb{D}^{1,2}$, where $c_u$ is just a constant depending on $u$. The divergence operator $\delta$ is also called Skorokhod integral. It is known that 
in the case of Brownian motion it coincides with Ito integral for adapted integrands. Hereafter for any $u \in \text{Dom}\delta$, we denote 
$\delta(u)= \int_{0}^{T} u_t \delta B_t$.\\

Following \cite{a-m-n} one can develop a Malliavin calculus for any continuous Gaussian process $G$ of the form 
\begin{equation*}
 G_t = \int_{0}^{t} K(t,s) \ud W_s
\end{equation*}
where $W$ is a Brownian motion and the kernel $K, \ 0<s<t<T$ satisfying $\sup_{t \in [0,T]} \int_{0}^{t}K(t,s)^{2} \ud s < \infty$. Consider the linear operator $K^{*}$ from $\mathcal{E}$ to $L^{2}[0,T]$ defined by
\begin{equation*}
 (K^{*} \varphi)(s) = \varphi(s)K(T,s) + \int_{s}^{T} \left[ \varphi(t) - \varphi(s) \right] K(\ud t,s).
\end{equation*}
The Hilbert space $\mathcal{H}$ generated by covariance function of the Gaussian process $G$ can be represented as $\mathcal{H} = (K^{*})^{-1} (L^{2}[0,T])$ and 
$\mathbb{D}^{1,2}_{G}(\mathcal{H}) =(K^{*})^{-1} \big(\mathbb{D}^{1,2}_{W}(L^{2}[0,T])\big)$. For any $n \ge 1$, let $\mathscr{H}_n$ be the $n$th Wiener chaos of $G$, i.e.
the closed linear subspace of $L^2 (\Omega)$ generated by the random variables 
$\{ H_n \left( G(\varphi) \right),\ \varphi \in \mathcal{H}, \ \Vert \varphi \Vert_{\mathcal{H}} = 1\}$, and $H_n$ is the $n$th Hermite polynomial. It is well known that
the mapping $I_{n}^{G}(\varphi^{\otimes n}) = n! H_n \left( G(\varphi)\right)$ provides a linear isometry between the symmetric tensor product $\mathcal{H}^{\odot n}$
and subspace $\mathscr{H}_n$.
\begin{mydef}\cite{a-m-n}\label{regular}
We say that the kernel $K$ is \textit{regular} if for all $s \in [0,T), \ K(\cdot,s)$ has bounded variation on the interval $(s,T]$, and 
\begin{equation*}
 \int_{0}^{T} \vert K \vert \big( (s,T],s\big)^{2} \ud s < \infty.
\end{equation*}
\end{mydef}
For regular kernel $K$, put $K(s^{+},s):= K(T,s) -  K \big( (s,T],s\big)$. For any $\varphi \in \mathcal{E}$, define the seminorm
\begin{equation*}
 \Vert \varphi \Vert_{Kr}^{2}=\int^T_0 \varphi(s)^2 K(s^+,s)^2 \ud s + \int^T_0 \left( \int^T_s \vert \varphi(t)\vert \vert K \vert(\ud t,s) \right) ^2 \ud s.
\end{equation*}
Denote by $\mathcal{H}_{Kr}$ the completion of $\mathcal{E}$ with respect to seminorm $\Vert \ , \Vert_{Kr}$. The following proposition establishes the relationship between path-wise integral and Skorokhod integral.
\begin{propo}\cite{a-m-n}\label{path-sko}
 Assume $K$ is a regular kernel with $K(s^{+},s)=0$ and $u$ is a process in $\mathbb{D}^{1,2}_{G}(\mathcal{H}_{Kr})$. Then the process $u$ is Stratonovich integrable with respect to $G$ and 
 \begin{equation*}
  \int^T_0 u_t \ud G_t = \int^T_0 u_t \delta G_t + \int^T_0 \left( \int^T_s D_s u_t K(\ud t, s) \right)\ud s.
 \end{equation*}
\end{propo}
The next proposition is taken from \cite{n-o} and provides a central limit theorem for a sequence of multiple Wiener integrals. Let $\mathcal{N}(0,\sigma^2)$ denote the Gaussian distribution with zero mean and variance $\sigma^2$.

\begin{propo}\label{CLT} Let $\{F_n\}_{n \ge 1}$ be a sequence of random variables in the $q$-th Wiener chaos, $q \ge2$, such that $\lim_{n \to \infty} \E(F_n ^2) = \sigma^2$. Then the following statements are equivalent:
 \begin{description}
 \item[(i)] $F_n$ converges in distribution to $\mathcal{N}(0,\sigma^2)$ as $n$ tends to infinity.
 
 \item[(ii)] $ \Vert DF_n \Vert^{2}_{\mathcal{H}}$ converges in $L^{2}(\Omega)$ to $q \sigma^{2}$ as $n$ tends to infinity.
 \end{description}
\end{propo}

The next proposition will be used to bridge our original Langevin equation $(\ref{Langevin})$ to a new SDE that plays the key role in our computations.

\begin{propo}\label{timechange} \cite{b-n}
 Let $B= \{ B_t \}_{t \in [0,T]}$ be a fractional Brownian motion with Hurst parameter $H \in (\frac{1}{2},1)$. Suppose $f:[0,T] \to \R$ be a strictly positive, 
absolutely continuous and $f^{'}$ is locally square integrable function. Then there exists a locally square integrable kernel $L$ such that
\begin{equation*}
B_{\int_{0}^{t} f(s)^{\frac{1}{H}} \ud s} \stackrel{\text{law}}{=} \int_{0}^{t} f(s) \ud B_s + \int_{0}^{t} \Big( \int_{r}^{t} f(u) \Big( \int_{r}^{u} \frac{\partial K_H}{\partial s} (u,s) L(s,r) \ud s \Big) \ud u \Big) \ud W_r
\end{equation*}
where Brownian motion $W$ is given by \ref{BW}. Moreover the kernel $L$ satisfies in
\begin{equation*}
 \phi(t,s)= \int_{s}^{t} f(t) \left( \frac{t}{r} \right)^{H - \frac{1}{2}} \left( t-r \right)^{H - \frac{3}{2}} L(r,s) \ud r
\end{equation*}
and
\begin{equation*}
\begin{split}
\phi(t,s)&=\frac{\partial K_{H}}{\partial t} \Big( \int_{0}^{t} f(v)^{\frac{1}{H}} \ud v, \int_{0}^{s} f(v)^{\frac{1}{H}} \ud v \Big) f(t)^{\frac{1}{H}} f(s)^{\frac{1}{2H}}\\
& \hspace{6cm } - f(t) \frac{\partial K_{H}}{\partial t} (t,s).
\end{split}
\end{equation*}
\end{propo}

\section{Main Results}\label{main}
For the rest of the paper, we assume that $H >\frac{1}{2}$. Let $X= \{ X_t \} _{t \in [0,T]}$ be the solution of the Langevin equation $(\ref{Langevin})$. Assume $X_0 = 0$ and $\theta >1$ $($this is a technical assumption which we need in the proof 
of the main result$)$. Then the solution can be expressed as
\begin{equation}\label{X}
 X_t = \int_0 ^ t e^{- \theta (t-s)} \ud Y^{(1)}_s.
\end{equation}

 




The next theorem gives the main result of the paper.
\begin{thm} \label{main_theorem}
 The least squares estimator $ \widehat\theta_T $ given by $(\ref{leastsquare})$ is weakly consistent, i.e. 
\begin{equation*}
\widehat\theta_T \rightarrow \theta
\end{equation*}
in probability, as T tends to infinity.
\end{thm}
The next theorem provides the rate of convergence for the least squares estimator towards the Gaussian distribution. Let $B(x,y)$ denote the complete Beta function and $\stackrel{law}{\to}$ stands for convergence in distribution.

\begin{thm}\label{second_main_theorem}
For the least squares estimator $ \widehat\theta_T $ given by $(\ref{leastsquare})$, we have
 \begin{equation*}
 \sqrt{T}\left(  \widehat\theta_T - \theta \right) \stackrel{law}{\to} \mathcal{N}(0,\sigma^{2}(\theta,H))
 \end{equation*}
as $T$ tends to infinity, where
\begin{equation*}
 \begin{split}
 \sigma^{2}(\theta,H)&= \frac{\theta^2 }{H^{2}B((\theta - 1)H +1,2H-1)^2}\\
 &\times 2 \int_{[0,\infty)^3}e^{-\theta y}e^{-\theta \vert z-x \vert} e^{(1- \frac{1}{H})(x+y+z)}
 \Big[ \big\vert e^{- \frac{y}{H}} - e^{- \frac{x}{H}} \big\vert \big\vert 1 - e^{- \frac{z}{H}} \big\vert \Big]^{2H-2} \ud x \ud y \ud z.
\end{split}
 \end{equation*}
\end{thm}

\begin{rmk}
It is worth to mention that the central limit theorem \ref{second_main_theorem} holds for whole range $H \in (\frac{1}{2},1)$, whereas for the fractional Ornstein-Uhlenbeck process of the first kind we have restriction 
$H \in [\frac{1}{2},\frac{3}{4})$, see \cite{h-n}. 
\end{rmk}

To prove Theorem \ref{main_theorem}, we need a series of lemmas. We start with the following lemma that provides an alternative stochastic differential equation of the Langevin equation (\ref{Langevin}).

\begin{lemma}\label{alternative_X_t}
Let $\tilde B_t= B_{t+H} - B_{H}$ be the shifted fractional Brownian motion. Then there exits a regular Volterra-type kernel $\tilde{L}$, in sense of Definition \ref{regular},  
so that for the solution of the following stochastic differential equation

\begin{equation}\label{diffeqXtilde}
\mathrm{d}\tilde{X}_t = -\theta \tilde{X}_t\,\mathrm{d}t+
\mathrm{d}\tilde{G}_t,\quad \tilde{X}_0 = 0,
\end{equation}

we have, $\{X_t \}_{t\in [0,T]} \stackrel{law}{=}\{\tilde{X}_t \}_{t\in [0,T]}$  where the Gaussian process

\begin{equation*}\label{Gtilde}
\tilde{G}_t=\int_0^t  \left(K_H(t,s)+\tilde{L}(t,s)\right)\,\mathrm{d}\tilde{W}_s
\end{equation*}

and Brownian motion $\tilde{W}$ is defined as \ref{BW}.
\end{lemma}

\begin{proof} [\textbf{Proof}] 
It is easy to check that the solution of the SDE $(\ref{diffeqXtilde})$ is given by
\begin{equation*}
\tilde{X}_t=\int_0^t e^{-\theta(t-s)}\,\ud \tilde{G}_s, \quad t\in[0,T].
\end{equation*}

Moreover, $\mathbb{E}(X_t)=\mathbb{E}(\tilde{X}_t)=0$, and 

\begin{dmath*}
\text{Cov}_X(t,s)=\mathbb{E}\left(\int_0^t e^{-\theta(t-u)}\,\ud Y^{(1)}_u 
\int_0^s e^{-\theta(s-v)}\,\ud Y^{(1)}_v\right)=\int_0^t\int_0^s e^{-\theta(t-u)-\theta(s-v)}\,
\text{Cov}_{Y^{(1)}}(\ud u,\ud v).
\end{dmath*}

Similarly

\begin{dmath*}
\text{Cov}_{\tilde{X}}(t,s)=\mathbb{E}\left(\int_0^t e^{-\theta(t-u)}\,\ud \tilde{G}_s 
\int_0^s e^{-\theta(s-v)}\,\ud \tilde{G}_v\right)=\int_0^t\int_0^s e^{-\theta(t-u)-\theta(s-v)}\,
\text{Cov}_{\tilde{G}}(\ud u,\ud v).
\end{dmath*}

Therefore, it is enough to show that for the Gaussian processes $Y^{(1)}$ and $\tilde{G}$, we have
\begin{equation*}
\text{Cov}_{Y^{(1)}}(s,t)=\mathbb{E}(Y^{(1)}_s Y^{(1)}_t)= \mathbb{E}(\tilde{G}_s\tilde{G}_t) = 
\text{Cov}_{\tilde{G}}(s,t).
\end{equation*}

To see this, consider the function $f(t)=e^{t}$ and apply proposition \ref{timechange}. Then 
\begin{dmath} \label{B_{a_t}}
 B_{a_t}=\tilde{B}_{\int_0^t e^{\frac{r}{H}}\,\mathrm{d}r} + B_{H} \stackrel{\text{law}}{=}
\int_0^t e^r\,\mathrm{d}\tilde{B}_r + \int_0^t g(t,r)\,\mathrm{d}\tilde{W}_r +B_H
\end{dmath}

where the Volterra-type Kernel $g$ is given by

\begin{equation*}
 g(t,r)=\int_r^t e^u \left(\int_r^u \frac{\partial K_H}{\partial s}(u,s)L(s,r)\,\mathrm{d}s\right)\mathrm{d}u.
\end{equation*}

Now, using integration by parts formula, we obtain

\begin{dmath}\label{Y-one}
 Y^{(1)}_t= \int_0^t e^{-s}\,\ud B_{a_s} \stackrel{\text{law}}{=} \int_0^t e^{-s} \,\ud
\left(\int_0^s e^r\,\mathrm{d}\tilde{B}_r + \int_0^s g(s,r)\,\mathrm{d}\tilde{W}_r +B_H \right)
= \tilde{B}_t + \int_0^t e^{-s} \, \ud \left(\int_0^s g(s,r)\,\mathrm{d}\tilde{W}_r \right)=
\tilde{B}_t + e^{-t} \int_0^t g(t,r)\,\mathrm{d}\tilde{W}_r + 
\int^t_0 \left(\int_0^s g(s,r)\,\mathrm{d}\tilde{W}_r\right) e^{-s}\,\ud s
\end{dmath}


Next, using Fubini theorem for Brownian motion $\tilde{W}$, we see that
\begin{equation*}
A= \int^t_0 \left(\int_0^s g(s,r)\,\mathrm{d}\tilde{W}_r\right) e^{-s}\,\ud s=
\int_0^t \int^t_r e^{-s} g(s,r)\,\ud s \,\ud \tilde{W}_r
\end{equation*}
Put

\begin{equation*}
 h(u,r) = \int_r^u \frac{\partial K_H}{\partial s}(u,v)L(v,r)\,\mathrm{d}v,
\end{equation*}

and using Fubini theorem for real-valued functions, we obtain

\begin{dmath*}
 \int_r^t e^{-s}g(s,r)\,\mathrm{d}s = \int^t_r e^{-s} \int_r^s e^u h(u,r)\,\mathrm{d}u\mathrm{d}s\\
= \int^t_r e^u h(u,r)
\left(\int_u^t e^{-s}\,\mathrm{d}s\right)\mathrm{d}u\\
=\int^t_r h(u,r)\,\mathrm{d}u - e^{-t}\int^t_r e^u h(u,r)\,\mathrm{d}u.
\end{dmath*}

Now, plug this into $A$, we infer that

\begin{dmath*}
A=\int_0^t\int_0^s e^{-s}g(s,r)\,\mathrm{d}s\mathrm{d}\tilde{W}_r =
\int^t_0 \left( \int^t_r h(u,r)\,\mathrm{d}u - 
e^{-t}\int^t_r e^u h(u,r)\,\mathrm{d}u \right)\,\mathrm{d}\tilde{W}_r \,{:=} A_1-A_2
\end{dmath*}

where

\begin{equation}\label{B_1}
 A_1= \int^t_0\int^t_r h(u,r)\,\mathrm{d}u\mathrm{d}\tilde{W}_r,
\end{equation}
and the term $A_2$ can be written as
\begin{equation}\label{B_2}
A_2 = \int^t_0 e^{-t} \int^t_r e^u h(u,r)\,\mathrm{d}u \,\mathrm{d}\tilde{W}_r =
\int^t_0 e^{-t} g(t,r)\,\mathrm{d}\tilde{W}_r.
\end{equation}

Finally, plug in (\ref{B_1}), (\ref{B_2}) and (\ref{representation FBW}) into (\ref{Y-one}), 
we obtain

\begin{dmath*}
 Y^{(1)}_t \stackrel{\text{law}}{=} \int^t_0 \left(K_H(t,s)+\int^t_r h(u,r)\,\mathrm{d}u\right)\,\mathrm{d}\tilde{W}_r.
\end{dmath*}
Therefore, it is enough to introduce the kernel $\tilde{L}$ as
\begin{equation}\label{kernel}
 \tilde{L}(t,r) = \int^t_r h(u,r)\,\mathrm{d}u.
\end{equation}
and then the result follows.
\end{proof}

The next lemma gives an alternative form for the least squares estimator in terms of the process $\tilde{X}$ and Gaussian noise $\tilde{G}$.
\begin{lemma}\label{thetatilde}
 For the least squares estimator $\widehat\theta_T$ we have
\begin{equation}
 \widehat\theta_T \stackrel{\text{law}}{=}\theta-\frac{\int^T_0 \tilde{X}_t\,\delta \tilde{G}_t}{\int^T_0 \tilde{X}^2_t\,\ud t}.
\end{equation}
\end{lemma}

\begin{proof}[\textbf{Proof}]
 Note that the least squares estimator $\widehat\theta_T$ is actually equal to 
\begin{equation*}
 \widehat\theta_T = \theta -\frac{\int^T_0 X_t\,\delta Y^{(1)}_t}{\int^T_0 X^2_t\,\ud t}.
\end{equation*}

On the other hand, we have

\begin{equation*}
 \int^T_0 X_t\,\delta Y^{(1)}_t = 
\int^T_0\int^t_0 e^{-\theta(t-s)}\,\delta Y^{(1)}_s\,\delta Y^{(1)}_t= I_2^{Y^{(1)}}(f) 
\end{equation*}

and the function $f$ is given by $f(t,s)=\frac{1}{2}e^{-\theta|t-s|}$. Similarly,

\begin{equation*}
\int^T_0 \tilde{X}_t\,\delta\tilde{G}_t  = 
\int^T_0\int^t_0 e^{-\theta(t-s)}\,\delta \tilde{G}_s\,\delta \tilde{G}_t= I_2^{\tilde{G}}(f).
\end{equation*}

This immediately implies that
\begin{equation*}
 \int^T_0 X_t\,\delta Y^{(1)}_t \stackrel{\text{law}}{=}\int^T_0 \tilde{X}_t\,\delta \tilde{G}_t.
\end{equation*}

Moreover, for every $t \in [0,T]$ there exists a function $g(\cdot,t)$ so that we can write 
\begin{equation*}
X_t=I_1^{Y^{(1)}}\left(g(\cdot,t)\right).
\end{equation*}
Using multiplication formula for multiple stochastic integrals, we obtain

\begin{dmath*}
X_t^2
= \lVert g(\cdot,t)\lVert_{\mathcal{H}}^2 + I_2^{Y^{(1)}}\left(g(\cdot,t) \tilde{\otimes} 
g(\cdot,t)\right)
\end{dmath*}

where $\mathcal{H}$ stands for Hilbert space generated by the covariance function of the Gaussian process $Y^{(1)}$. Therefore, using Fubini theorem

\begin{dmath*}
\int^T_0 X^2_t\,\ud t= \int^T_0 \lVert g(\cdot,t)\lVert_{\mathcal{H}}^2 \,\ud t+ 
I_2^{Y^{(1)}}\left(\int^T_0 \left(g(\cdot,t) \tilde{\otimes} g(\cdot,t)\right)\,\ud t\right)
\end{dmath*}

Similarly, in the same lines one can obtain that

\begin{dmath*}
\int^T_0 \tilde{X}^2_t\,\ud t= \int^T_0 \lVert g(\cdot,t)\lVert_{\tilde{\mathcal{H}}}^2 \,\ud t+ 
I_2^{\tilde{G}}\left(\int^T_0 \left(g(\cdot,t) \tilde{\otimes} g(\cdot,t)\right)\,\ud t\right)
\end{dmath*}

where $\tilde{\mathcal{H}}$ stands for Hilbert space generated by the covariance function of the Gaussian process $\tilde{G}$. Note that $\lVert g(\cdot,t)\lVert_{\mathcal{H}} = \lVert g(\cdot,t)\lVert_{\tilde{\mathcal{H}}}$. Consequently
\begin{equation*}
\int^T_0 X^2_t\,\ud t \stackrel{\text{law}}{=}
\int^T_0 \tilde{X}^2_t\,\ud t.
 \end{equation*}

Finally, by Lemma \ref{correlations} in Appendix, we can deduce that for the following two dimensional random vectors we have

\begin{dmath*}
\left( \int^T_0 X_t\,\delta Y^{(1)}_t , \int^T_0 X^2_t\,\ud t \right) 
\stackrel{\text{law}}{=}\left( \int^T_0 \tilde{X}_t\,\delta \tilde{G}_t ,\int^T_0 \tilde{X}^2_t\,\ud t \right).
\end{dmath*}
We conclude that
\begin{equation*}
\frac{\int^T_0 X_t\,\delta Y^{(1)}_t}{\int^T_0 X^2_t\,\ud t} \stackrel{\text{law}}{=}
\frac{\int^T_0 \tilde{X}_t\,\delta \tilde{G}_t}{\int^T_0 \tilde{X}^2_t\,\ud t}.
\end{equation*}
\end{proof}

In the next lemma, we give an alternative expression of the estimator $\widehat\theta_T$.
\begin{lemma} \label{estimator}
The least squares estimator $\widehat\theta_T$ can be written as 
\begin{dmath*}
\widehat\theta_T \stackrel{\text{law}}{=} - \frac{\frac{1}{2}\tilde{X}_T^2}{\int^T_0 \tilde{X}^2_t\,\ud t}+
\frac{\int^T_0\int^{T}_s D_s^{(\tilde W)} \tilde{X}_t\: 
\left( K_H(\ud t,s)+\tilde{L}(\ud t, s)\right)\,\ud s}{\int^T_0 \tilde{X}^2_t\,\ud t},
\end{dmath*}
where $D^{(\tilde W)}$ denotes the Malliavin derivative operator with respect to Brownian motion $\tilde W$.
\end{lemma}
\begin{proof}[\textbf{Proof}]
Using Lemma \ref{thetatilde}, Proposition \ref{path-sko} and SDE $(\ref{diffeqXtilde})$, we can infer that 
\begin{dmath*}
 \widehat\theta_T \stackrel{\text{law}}{=}\theta-\frac{\int^T_0 \tilde{X}_t\, \delta  \tilde{G}_t}{\int^T_0 
\tilde{X}^2_t\,\ud t} 
= \theta- \frac{\int^T_0 \tilde{X}_t\,\ud 
\tilde{G}_t-\int^T_0\int^{T}_s D_s^{(\tilde W)} \tilde{X}_t\: 
\left( K_H(\ud t,s)+\tilde{L}(\ud t, s)\right)\,\ud s}{\int^T_0 \tilde{X}^2_t\,\ud t} 
= \theta- \frac{\int^T_0 \tilde{X}_t\,
(\mathrm{d}\tilde{X}_t + \theta \tilde{X}_t\,\mathrm{d}t)-\int^T_0\int^{T}_s D_s^{(\tilde W)} \tilde{X}_t\: 
\left( K_H(\ud t,s)+\tilde{L}(\ud t, s)\right)\,\ud s}{\int^T_0 \tilde{X}^2_t\,\ud t} 
= -\frac{\frac{1}{2} \tilde{X}_T^2}{\int^T_0 \tilde{X}^2_t\,\ud t}+
\frac{\int^T_0\int^{T}_s D_s^{(\tilde W)} \tilde{X}_t\: 
\left( K_H(\ud t,s)+\tilde{L}(\ud t, s)\right)\,\ud s}{\int^T_0 \tilde{X}^2_t\,\ud t}.
\end{dmath*}
\end{proof}

In order to prove the theorem, we need the following lemma. The next lemma can be used for an alternative estimator of the parameter $\theta$.

\begin{lemma}\label{denominator}
 For the process $X$ given by $(\ref{X})$ and consequently for the process $\tilde{X}$, as $T$ tends to infinity, we have 
\begin{equation*} 
 \frac{1}{T}\int^T_0 X^2_t\,\ud t \rightarrow \frac{(2H-1)H^{2H}}{\theta}\:B\left( (\theta-1)H+1, 2H-1\right)
\end{equation*}
almost surely and in $L^2$.
\end{lemma}

\begin{proof}[\textbf{Proof}]
Note that for every $t \geq 0$, we have
\begin{equation} \label{Y_t}
U_t = e^{-\theta t}\int^t_{-\infty} e^{(\theta-1)s}\,\ud B_{a_s} = X_t+e^{-\theta t}\xi, 
\end{equation}

where $\xi=\int^0_{-\infty} e^{(\theta-1)s} \,\ud B_{a_s}$. For the stationary Gaussian process $U$ we have $($see Proposition 3.11 of \cite{k-s}$)$ 

\begin{equation*}
 \mathbb{E}\left( U_{t} U_{0} \right) = O (e^{-\theta t})\rightarrow 0,
\end{equation*}

as $t$ tends to infinity. Hence, $U$ is ergodic. Therefore, the ergodic theorem implies that

\begin{equation*}
 \frac{1}{T}\int^T_0  U_t ^{2}\,\ud t \rightarrow \mathbb{E}( U_0^{2} )
\end{equation*}

as $T$ tends to infinity, almost surely and in  $L^2$. It is straightforward from (\ref{Y_t}) to check that

\begin{equation} \label{denominator_1}
 \frac{1}{T}\int^T_0 X^2_t\,\ud t \rightarrow \mathbb{E}  (U_{0}^2).
\end{equation}

On the other hand, using change of variable we see that

\begin{dmath*}
 U_t=e^{-\theta t}\int^t_{-\infty} e^{(\theta-1)s}\,\ud B_{a_s}=H^{-(\theta-1)H} e^{-\theta t} \int^{a_t}_0
s^{(\theta-1)H}\,\ud B_s.
\end{dmath*}

Therefore, using formula $(\ref{isometryfbm})$, we obtain

\begin{dmath} \label{denominator_2}
\mathbb{E} (U_{0})^2= H^{-2(\theta-1)H} \mathbb{E}\left( \int_0^{a_0} s^{(\theta-1)H}\,\ud B_s\right)^2
=  H^{-2(\theta-1)H}H(2H-1)\int_0^{a_0}\int_0^{a_0}s^{(\theta-1)H}t^{(\theta-1)H}|s-t|^{2H-2}\,\ud s\ud t
=  \frac{(2H-1)H^{2H}}{\theta} B\left( (\theta-1)H +1, 2H-1\right)
\end{dmath}

Substituting (\ref{denominator_2}) into (\ref{denominator_1}) yields the result.
\end{proof}

\begin{proof}[\textbf{Proof of Theorem \ref{main_theorem}}]

From lemma (\ref{estimator}), we can write $\widehat\theta_T=I_1+I_2$. By Lemma \ref{denominator}, we know
that $\frac{1}{T}\int^T_0 \tilde{X}^2_t\,\ud t$ converges almost surely and in $L^2$, as $T$ tends to infinity to 
a constant. Then, it suffices to study the almost sure convergence of the numerators of $I_1$ and $I_2$.\\

\textit{Step 1: $I_1\rightarrow 0$ almost surely as T tends to infinity.} 
From Lemma $(\ref{Ua.s.})$ in Appendix and the assumption $\theta>1$, we deduce that almost surely
\begin{equation}
 \lim_{T\rightarrow\infty} \frac{\tilde{X}_T^2}{T} = 0.
\end{equation}
\\
\textit{Step 2: $I_2\rightarrow \theta$ almost surely as T tends to infinity.} \\

First note that the Malliavin derivative $D^{(\tilde W)}_s \tilde{X}_t$ with respect to Brownian motion $\tilde{W}$ 
is given by $($see Section 2.2$)$ 

\begin{equation*}
D^{(\tilde W)}_s \tilde{X}_t =\int^t_s e^{-\theta (t-u)}
 \left(\frac{\partial K_H}{\partial u}(u,s)+h(u,s)\right)\,\ud u .
\end{equation*}

Also,

\begin{dmath*}
 K_H(\ud t,s)+\tilde{L}(\ud t,s)=\left(\frac{\partial K_H}{\partial t}
(t,s)+h(t,s)\right)\,\ud t = \left(\frac{\partial K_H}{\partial t}
(t,s)+e^{-t} \phi(t,s)\right)\,\ud t 
= e^{-t}\frac{\partial K_H}{\partial t}
\left( \int^t_0 e^{\frac{u}{H}}\,\ud u,\int^s_0 e^{\frac{u}{H}}\,\ud u \right)e^{\frac{t}{H}}
e^{\frac{s}{2H}}\, \ud t.
\end{dmath*}

Denote the numerator of the term $I_2$ with $ I_2^{up}$. Therefore,

\begin{dmath*}
 I_2^{up}=\frac{1}{T}\int^T_0\int^{T}_s D_s^{(\tilde W)} \tilde{X}_t\: 
\left( K_H(\ud t,s)+\tilde{L}(\ud t, s)\right)\,\ud s 
=\frac{1}{T}\int^T_0\int^{T}_s \int^t_s e^{- \theta (t-u)} e^{-u} \frac{\partial K_H}{\partial u}
\left( \int^u_0 e^{\frac{u}{H}}\,\ud u,\int^s_0 e^{\frac{u}{H}}\,\ud u \right) e^{\frac{u}{H}}
e^{\frac{s}{2H}}\,\ud u \times e^{-t}\frac{\partial K_H}{\partial t}
\left( \int^t_0 e^{\frac{u}{H}}\,\ud u,\int^s_0 e^{\frac{u}{H}}\,\ud u \right)e^{\frac{t}{H}}
e^{\frac{s}{2H}}\, \ud t \, \ud s.
\end{dmath*}

Now, by Fubini theorem for real-valued functions and the Lemma $(\ref{I_{2}^{up}})$ in Appendix, we obtain

\begin{dmath*}
I_2^{up}= \frac{1}{T} \int^T_0 \int^t_0 e^{-\theta(t-u)} e^{-u}e^{\frac{u}{H}} e^{-t} e^{\frac{t}{H}} 
\times \int^{t\wedge u}_0 \frac{\partial K_H}{\partial t}
\left( \int^t_0 e^{\frac{v}{H}}\,\ud v,\int^s_0 e^{\frac{v}{H}}\,\ud v \right) \frac{\partial K_H}{\partial u}
\left( \int^u_0 e^{\frac{v}{H}}\,\ud v,\int^s_0 e^{\frac{v}{H}}\,\ud v \right)\, e^{\frac{s}{H}} \ud s \, \ud u\,\ud t 
= \frac{\alpha_{H}}{T}  \int^T_0 \int^t_0 e^{-\theta(t-u)} e^{-u}e^{\frac{u}{H}} e^{-t} e^{\frac{t}{H}}
\left| \int^u_0 e^{\frac{v}{H}}\,\ud v-\int^t_0 e^{\frac{v}{H}}\,\ud v \right|^{2H-2}\,\ud u\,\ud t     
\end{dmath*}

We make a change of variables $p =\int^u_0 e^{\frac{v}{H}}\,\ud v$ and 
$q =\int^t_0 e^{\frac{v}{H}}\,\ud v$. Hence,

\begin{dmath*}
 I^{up}_2= \frac{\alpha_H}{T} \int^{a_T-H}_0 \int^q_0 \left( 1+\frac{q}{H}\right)^{-H(1+\theta)}
\left(  1+\frac{p}{H}\right)^{H(\theta-1)} |p-q|^{2H-2}\,\ud p\,\ud q =
\frac{\alpha_H H^{2H}}{T} \int^{a_T-H}_0 \int^q_0 \left( q+H\right)^{-H(1+\theta)}
\left(  p+H\right)^{H(\theta-1)} |p-q|^{2H-2}\,\ud p\,\ud q =
\frac{\alpha_H H^{2H}}{T} \int^{a_T}_{a_0} \int^{q+H}_H q^{-H(1+\theta)} p^{H(\theta-1)} 
|p-q|^{2H-2}\,\ud p\,\ud q =
\frac{\alpha_H H^{2H}}{T} \int^{a_T}_{a_0} \frac{1}{q}\int^{1+\frac{H}{q}}_{\frac{H}{q}} 
 y^{(\theta-1)H} (1-y)^{2H-2} \,\ud y\,\ud q
\end{dmath*}

As a result,

\begin{equation*}
 I^{up}_2 \longrightarrow (2H-1)H^{2H} B\left((\theta-1)H+1, 2H-1 \right)
\end{equation*}
and the claim follows.
\end{proof}

\begin{proof}[\textbf{Proof of Theorem \ref{second_main_theorem}}]
Using Lemma \ref{thetatilde}, we have
\begin{equation*}
 \sqrt{T} \left( \widehat{\theta}_{T} - \theta \right) \stackrel{law}{=} - \frac{\sqrt{T} I_{2}^{\tilde{G}}\left( \frac{1}{2} e^{ - \theta \vert t-s \vert} \right)}{ \int^T_0 \tilde{X}^2_t\,\ud t} = - \frac{F_T}{\frac{1}{T}\int^T_0 \tilde{X}^2_t\,\ud t},
\end{equation*}
where $F_T$ stands for the double stochastic integral
\begin{equation*}
F_T = \frac{1}{\sqrt{T}} I_{2}^{\tilde{G}}\left( \frac{1}{2} e^{ - \theta \vert t-s \vert} \right).
\end{equation*}
Therefore, taking into account Lemma \ref{denominator}, it is enough to show that the sequence $\{ F_T\}$ for $T=1,2, \cdots$ converges in law to a Gaussian distribution. To show this, we use Proposition \ref{CLT}. We have 
\begin{equation*}
\begin{split}
\E (F_T)^2 &= \frac{\alpha_{H}^2}{2 T} \int_{[0,T]^4} \Big[ e^{-\theta \left( \vert u_1 - v_1 \vert + \vert u_2 - v_2 \vert \right)} \ e^{(\frac{1}{H} - 1)(u_1 + v_1 + u_2 + v_2)}\\
& \times \Big\vert \int_{0}^{u_2} e^{\frac{x}{H}} \ud x - \int_{0}^{u_1} e^{\frac{x}{H}} \ud x\Big\vert^{2H-2}  \Big\vert \int_{0}^{v_2} e^{\frac{x}{H}} \ud x - \int_{0}^{v_1} e^{\frac{x}{H}} \ud x\Big\vert^{2H-2} \Big]  \ud u_1 \ud u_2 \ud v_1 \ud v_2\\
& := \frac{\alpha_{H}^2}{2 T} H^{4H-4} I_T,
\end{split}
\end{equation*}
where
\begin{equation*}
\begin{split}
I_T&=  \int_{[0,T]^4} \Big[ e^{-\theta \left( \vert u_1 - v_1 \vert + \vert u_2 - v_2 \vert \right)} \ e^{(\frac{1}{H} - 1)(u_1 + v_1 + u_2 + v_2)}\\
& \times \Big\vert  e^{\frac{u_2}{H}} - e^{\frac{u_1}{H}} \Big\vert^{2H-2}  \Big\vert  e^{\frac{v_2}{H}}  - e^{\frac{v_1}{H}} \Big\vert^{2H-2}\Big] \ud u_1 \ud u_2 \ud v_1 \ud v_2.
\end{split}
\end{equation*}
Taking the derivative with respect to $T$, change of variables $ x=T-u_1, y=T-u_2, z=T-v_1$, and taking limit as $T$ tends to infinity, we obtain that
\begin{equation*}
 \begin{split}
  \lim_{T \to \infty} \frac{\ud I_T}{\ud T}  = 4 \int_{[0,\infty)^3} & e^{-\theta y}e^{-\theta \vert z-x \vert} e^{(1- \frac{1}{H})(x+y+z)}\\
 & \times \Big[ \big\vert e^{- \frac{y}{H}} - e^{- \frac{x}{H}} \big\vert \big\vert 1 - e^{- \frac{z}{H}} \big\vert \Big]^{2H-2} \ud x \ud y \ud z
 \end{split}
\end{equation*}
and this integral is finite. Indeed using obvious bound $ e^{-\theta \vert z-x \vert} \le 1 $ and change of variables $ e^{- \frac{x}{H}}=a, e^{- \frac{y}{H}}=b, e^{- \frac{z}{H}}=c$, we infer that, it is smaller than
\begin{equation*}
 \begin{split}
 C(H)& B(1-H,2H-1) \int_{[0,1]^2} (ab)^{-H} b^{\theta H} \vert b -a \vert^{2H-2} \ud a \ud b \\
 & =C(H) B(1-H,2H-1) \Big[ \int_{0}^{1} \int_{0}^{b} + \int_{0}^{1} \int_{b}^{1} (ab)^{-H} b^{\theta H} \vert b -a \vert^{2H-2} \ud a \ud b \Big] \\
 &= C(H) B(1-H,2H-1)\left( J_1 + J_2 \right).
 \end{split}
\end{equation*}
For example, for the term $J_2$ we have
\begin{equation*}
 \begin{split}
 J_2 &= \int_{0}^{1} b^{\theta H -1}\ud b \int_{0}^{1} w^{(\theta -1) H} (1- w)^{2H-2} \ud w \\
 & = B(1 + (\theta -1)H , 2H-1) \int_{0}^{1} a^{\theta H -1}\ud a < \infty.
\end{split}
 \end{equation*}
With similar computation $($see \cite{a-v}, appendix$)$, one can show that
\begin{equation*}
\E \Big[ \Vert D_{s}F_{T}\Vert_{\tilde{\mathcal{H}}}^{2} - \E  \Vert D_{s}F_{T}\Vert_{\tilde{\mathcal{H}}}^{2} \Big]^{2} \to 0 \ \text{as} \ T \to \infty.
\end{equation*}
Taking into account that 
\begin{equation*}
\lim_{T \to \infty} \E  \Vert D_{s}F_{T}\Vert_{\tilde{\mathcal{H}}}^{2}= 2 \lim_{T \to \infty} \E(F_{T}^{2}),
\end{equation*}
we complete the proof.
\end{proof}

\begin{rmk}
Note that if one can replace Skorokhod integral with path-wise Riemann-Stieltjes integral in the formula of least squares estimator, then we can obtain the new estimator 
\begin{equation*}
 \widehat {\theta_{T}^{'}}:=- \frac{\int^T_0 X_t \ud X_t}{\int^T_0 X^2 _t \ud t} \to 0 
\end{equation*}
almost surely by Lemma \ref{denominator} and \ref{Ua.s.} in Appendix, as $T \to \infty$. 
\end{rmk}



\section{appendix}

\begin{lemma}\label{Ua.s.}
 For the stationary Gaussian process $U$ given by $(\ref{Y_t})$ and any $\alpha > 0$, as $T$ tends to infinity, we have 
\begin{equation*}
 \frac{U_T}{T^{\alpha}} \longrightarrow 0 \quad \mathrm{\text{almost surely}}.
\end{equation*}
\end{lemma}
\begin{proof}
Using Proposition 3.11 of \cite{k-s}, there exists a constant $C(H,\theta)$ such that
\begin{equation*}\label{lemma1-1}
c(t):= \text{Cov}(U_t,U_0)= C(H,\theta) e^{- \theta t} \left( \int^{a_{t}}_0 \int^{a_{0}}_0 (xy)^{(\theta - 1)H} \vert x-y \vert^{2H-2} \ud x \ud y \right),
\end{equation*}
and the constant $C(H,\theta)= H(2H-1) H^{2H(1 - \theta)}$. Denote the term inside parentheses by $I(t)$. Then with some direct computations one can see

\begin{equation*}
 \begin{split}
  I(t)&= \int_{0}^{a_t} \int_{0}^{\frac{a_0}{a_t}y} (xy)^{(\theta - 1) H} (y - x)^{2H-2}  \ud x \ud y\\
  &+ \int_{0}^{a_0} \int_{0}^{x} (xy)^{(\theta - 1) H} ( x - y)^{2H-2}  \ud y \ud x\\
  &+ \int_{0}^{a_0} \int_{\frac{a_0}{a_t} y}^{y} (xy)^{(\theta - 1) H} (y - x)^{2H-2}  \ud x \ud y\\
  &= \frac{a_{0}^{2\theta H}}{ \theta H} B((\theta - 1)H +1, 2H-1)+ \frac{1}{2\theta H} ( a_{t}^{2 \theta H} -  a_{0}^{2 \theta H} ) \int_{0}^{\frac{a_0}{a_t}} z^{(\theta - 1)H} (1 - z)^{2H -2} \ud z.
 \end{split}
\end{equation*}
Therefore, as $t \to 0^+$, we have 
\begin{equation}\label{lemma1-1}
 \begin{split}
 c(t)&= \frac{(2H-1) H^{2H}}{\theta} B((\theta - 1)H +1, 2H-1) \\
 &- (2H-1)H^{2H} \times t \times \int_{\frac{a_0}{a_t}}^{1} z^{(\theta - 1)H} (1 - z)^{2H -2} \ud z + o(t^{2H}).
 \end{split}
\end{equation}

Let $B_{\text{inc}}(p,q,x)$ and $ _{2}{F}_{1}(a,b,c;x)$ stand for incomplete Beta function and Gauss Hyper geometric function respectively $($see \cite{te} for definition$)$. Using a relation between these two functions $($\cite{te}, page\ $289)$,
 we obtain that as $t \to 0^+$, we have

\begin{equation}\label{lemma1-2}
 \begin{split}
 \int_{\frac{a_0}{a_t}}^{1}  z^{(\theta - 1)H} & (1 - z)^{2H -2} \ud z= B_{\text{inc}}(2H-1,(\theta-1)H, 1 - \frac{a_0}{a_t})\\
 &= \frac{H H^{-2H}}{2H-1} e^{- (\theta +1) t} (a_t - a_0)^{2H-1} \  _{2}{F}_{1} ((\theta +1)H, 1, 2H; 1 - \frac{a_0}{a_t})\\
 &= \frac{H H^{-2H}}{2H-1} t^{2H-1} + o(t^{2H-1}).
 \end{split}
\end{equation}

Substituting $(\ref{lemma1-2})$ in $(\ref{lemma1-1})$, we obtain that as $t \to 0^+$, we have
\begin{equation*}
c(t)= \frac{(2H-1) H^{2H}}{\theta} B((\theta - 1)H +1, 2H-1) - H t^{2H} + o(t^{2H}).
\end{equation*}
Now the claim follows from Theorem 3.1 of \cite{pi}.
\end{proof}

\begin{lemma}\label{correlations}
 Set $U=I_2^{Y^{(1)}}(f)$, $V=I_2^{Y^{(1)}}\left(\int^T_0 \left(g(\cdot,t) \tilde{\otimes} 
g(\cdot,t)\right)\,\ud t\right)$ and $\tilde{U}=I_2^{\tilde{G}}(f)$, $\tilde{V}=I_2^{\tilde{G}}\left(\int^T_0 \left(g(\cdot,t) \tilde{\otimes} 
g(\cdot,t)\right)\,\ud t\right)$. Then

\begin{dmath*}
\mathbb{E}\left( U^{m_1} V^{m_2} \right) = \mathbb{E}\left(\tilde{U}^{m_1} 
\tilde{V}^{m_2}\right) \quad {\forall m_1,m_2 \geq 1}.
\end{dmath*}

\end{lemma}

\begin{proof}[\textbf{Proof}]
 It immediately follows from Corollary 7.3.1 of \cite{pe-t}.
\end{proof}

The next lemma is used in study of the term $I^{up}_2$.

\begin{lemma}\label{I_{2}^{up}}
 \begin{equation*}
\begin{split}  
\int^{s\wedge t}_0 \frac{\partial K_H}{\partial t}
\left(\int^t_0 e^{\frac{v}{H}}\,\ud v,
\int^u_0 e^{\frac{v}{H}}\,\ud v\right)& \frac{\partial K_H}{\partial s}
\left(\int^s_0 e^{\frac{v}{H}}\,\ud v,
\int^u_0 e^{\frac{v}{H}}\,\ud v\right) e^{\frac{u}{H}}\,\ud u \\
& = H(2H-1)\left| \int_{0}^{t} e^{\frac{v}{H}}  - \int_{0}^{s} e^{\frac{v}{H}} \, \ud v \right|^{2H-2}.
\end{split}
\end{equation*}
\end{lemma}

\begin{proof}
Put $V_t=B_{\int^t_0 e^{\frac{v}{H}}\,\ud v} $. Note that we can also represent this random variable as Wiener integral in the following way:
\begin{dgroup*}
\begin{dmath*}
V_t=\int^t_0 K_H\left( \int^t_0 e^{\frac{v}{H}}\, \ud v, \int^u_0 e^{\frac{v}{H}} \,\ud v\right)e^{\frac{u}{2H}}\,\ud W_u
\end{dmath*}
\end{dgroup*}
where Brownian motion $W$ is given by $(\ref{BW})$. Therefore
\begin{dmath}\label{covariance_XY}
\mathbb{E}\left( V_s V_t\right)= \int^{s\wedge t}_0 K_H\left( \int^s_0 e^{\frac{v}{H}}\, \ud v, \int^u_0 e^{\frac{v}{H}}\,\ud v\right) K_H\left( \int^t_0 e^{\frac{v}{H}}\, \ud v, \int^u_0 e^{\frac{v}{H}}\,\ud v\right)e^{\frac{u}{H}}\,\ud u.
\end{dmath}

On the other hand,

\begin{dmath}\label{covariance_fbm}
\mathbb{E}\left( V_s V_t\right)=\mathbb{E}\left( B_{\int^s_0 e^{\frac{v}{H}}\,\ud v}  B_{\int^t_0 e^{\frac{v}{H}}\,\ud v}\right)=
R_{H} \Big(\int^s_0 e^{\frac{v}{H}}\,\ud v, \int^t_0 e^{\frac{v}{H}}\,\ud v \Big).
\end{dmath}

With differentiating respect to $s$ and $t$ of the right hand side of (\ref{covariance_fbm}) and (\ref{covariance_XY}), we get the desired identity.

\end{proof}

\section{Acknowledgement}
Ehsan Azmoodeh thanks the Magnus Ehrnrooth foundation for financial support of the major part of the this work that had been done in Helsinki. 
Jos\'e Igor Morlanes thanks Prof. Esko Valkeila for his advice and his financial support via the Academy grant 21245. Both authors thank Tommi Sottinen and Lauri Viitasaari for useful comments and discussions. 

\addcontentsline{toc}{chapter}{\numberline{}Bibliography}

\end{document}